\documentclass{amsart}
\usepackage{xcolor}

\usepackage{mathtools}

\usepackage{hyperref}

\newcommand{\mc}[1]{\mathcal{#1}}
\newcommand{\wh}[1]{\widehat{#1}}

\theoremstyle{definition}
\newtheorem{definition}{Definition}[section]

\newtheorem{remark}[definition]{Remark}
\newtheorem{example}[definition]{Example}
\numberwithin{equation}{section}

\theoremstyle{plain}
\newtheorem{theorem}[definition]{Theorem}
\newtheorem{lemma}[definition]{Lemma}
\newtheorem{corollary}[definition]{Corollary}
\newtheorem{proposition}[definition]{Proposition}
\newtheorem{claim}[definition]{Claim}

\DeclareMathOperator{\Homeo}{Homeo}
\newcommand{\N}{\mathbb{N}}
\newcommand{\E}{\mathbb{E}}
\def\Diff{{\rm Diff}}
\newcommand{\vol}{\operatorname{vol}}
\def\cH{\mathcal{H}}
\def\bmu{\boldsymbol{\mu}}
\def\cD{\mathcal{D}}
\def\reals{\mathbb{R}}
\def\EXP{{\mathbb{E}}}
\def \qed{\hfill$\square$}
\def \R{\mathbb{R}}
\def\bbD{\mathbb{D}}
\def\Prob{{\mathbb{P}}}
\def\fN{\mathfrak{N}}
\def\SL{\mathop{\rm SL}}
\newcommand{\abs}[1]{\left|#1\right|}
\newcommand{\pez}[1]{\left(#1\right)}
\def\eps{{\varepsilon}}
\def\Var{{\rm Var}}
\def\cN{\mathcal{N}}
\def\tA{{\tilde A}}
\def\cA{\mathcal{A}}

\def\Cor{{\rm Cor}}

\def\gauss{{\mathfrak{f}}}
\def\DS{\displaystyle}

\usepackage[backend=biber,sorting=nty,sortcites=true,style=alphabetic,
doi=true,natbib=true,isbn=false,url=false,maxbibnames=99]{biblatex}
\renewbibmacro{in:}{}

\title[Quenched and annealed statistical properties for random systems]
{On equivalence of quenched and annealed statistical properties for conservative IID random dynamical systems}
\author{Jonathan DeWitt}
\address{Department of Mathematics, The Pennsylvania State University, State College, PA 16802, USA}
\email{dewitt@psu.edu}
\author{Dmitry Dolgopyat}
\address{Department of Mathematics, The University of Maryland, College Park, MD 20742, USA}
\email{dolgop@umd.edu}
\date{\today}

\begin{document}

\begin{abstract}
In this paper, we prove several theorems relating annealed exponential mixing of the two-point motion with quenched properties of the one-point motion for conservative IID random dynamical systems. In particular, we show that annealed exponential mixing of the two-point motion implies quenched exponential mixing of the one-point motion. 
We also show that 
if the two-point motion satisfies annealed exponential mixing and the annealed central limit theorem with polynomial rate of convergence, then the one-point motion satisfies a quenched CLT. These results hold for all H\"older and Sobolev spaces of positive index.
\end{abstract}

\maketitle

\section{Introduction}
\subsection{Quenched and annealed results}

Let $M$ be a closed Riemannian manifold and consider the IID sequence of diffeomorphisms distributed according to a measure $\mu$ on
$\Diff_{\vol}^\infty(M)$ with compact support. 
The random dynamics is then driven by the product measure $\mu^{\mathbb{N}}$ on $(\Diff^{\infty}_{\vol}(M))^{\mathbb{N}}$, and a particular realization of the random dynamics is given by a word $\omega\in (\Diff^{\infty}_{\vol}(M))^{\mathbb{N}}$. In a slight abuse of notation, let $F^N=F^N_\omega=f_{\sigma^{N-1}\omega}\circ \dots \circ f_{\omega}$. 
We also consider the two point motion, which is the induced action $F_{2}^N$ on $M\times M$ given by $F^N_{\omega}\times F^N_{\omega}$, i.e.~$F_{2}^N(x, y)=(F^N(x), F^N(y))$. Below we will often suppress the dependence on $\omega$. 
Denote by $\cH^p_0(M)$  the space of zero mean functions that belong to 
the Sobolev space of index $p$.

For random systems, there are two basic versions of each limit theorem: quenched and annealed. In a quenched limit theorem, one shows that for a.e.~realization $\omega$ of the random system that the limit theorem holds. In an annealed limit theorem, one additionally averages over the entire ensemble of possible realizations $\omega$. Naturally, in the quenched case, one often wants an additional estimate on the set of $\omega$ where the limit converges slowly. 

The goal of this note is to provide sufficient conditions for quenched exponential mixing and the quenched central limit theorem to follow from the annealed versions of these theorems, which are in principal easier to prove. The conditions that appear in this paper are not dynamical, and hence should be widely applicable. For example, none of our results make any particular assumption, on the transfer operator. The results in this paper show that the central limit theorem obtained our previous work \cite{dewitt2025conservative} also holds in the quenched setting.
The main results of this paper are Theorem \ref{ThQMix}, which states that annealed exponential mixing of the two point motion  implies quenched exponential mixing of the one point motion,
and Theorem \ref{ThAnQCLT}, which shows that annealed exponential mixing of the two point motion plus the annealed Central Limit Theorem for the two-point motion entails the quenched CLT. We also show when quenched results provide annealed results, and give some counterexamples showing the sharpness of our statements.

We follow an approach of \cite{DKK04}. However, our results are stronger because \cite{DKK04} considered a fixed
observable while we obtain results which are valid for all sufficiently smooth functions. This extension requires
additional work. The strengthened form of the results of \cite{DKK04} concerning the quenched exponential mixing
was also shown in \cite[Sec.~7]{bedrossian2022almost} in the context of random fluid dynamics on $\mathbb{T}^2$. Although the central limit theorem is not considered in that paper, our Theorem \ref{ThAnQCLT} suggests that the quenched central limit theorem should also hold in that context. 
Like \cite{DKK04}, other papers also consider the quenched central limit theorem but in a weaker sense: they first fix the function and \emph{then} show that for that particular function almost every realization satisfies the central limit theorem. The quenched central limit theorem for products of random toral automorphisms satisfying a cone condition was shown in \cite{ayyer2009quenched}, although in that paper they only consider the case of $C^{\infty}$ observables. In \cite{conze2012central} a quenched central limit theorem is obtained for certain products of toral automorphisms. In \cite[Thm.~7.1]{aimino2015annealed}, the authors say that the argument of \cite{ayyer2009quenched} can be adapted to their setting although they do not appear to give any general statement. For comparison, we remark that our Theorem \ref{ThAnQCLT}, allows any polynomial rate of convergence, rather than $N^{1/2}$. We have taken pains here to not include spectral gap as a hypothesis in our theorems. 

 It is perhaps natural to wonder why one might ask for annealed exponential mixing of the two point motion, rather than just the one point motion when looking for quenched exponential mixing. The reason is that if  $\E[\int A(B\circ F^N)\,dx]$ is small, then one has no control over what a particular realization of $\int A(B\circ F^N)\,dx$ might be: There might just be cancellation between the different terms. Hence to get control, it is natural to look at the variance of $\int A(B\circ F^N)\,dx$ and try to get control of that. But the variance can  be rewritten as
\[
\E\left[\left(\int_M AB\circ F^N_{\omega}\,dx\right)^2\right]=\E\left[\int_{M\times M} A(x)B(F^N_{\omega}(x))A(y)B(F^N_{\omega}(y))\,dx\,dy\right].
\]
Hence we are led directly to consider annealed mixing of the two point motion. Studying the quantity above is enough to show that for fixed functions a.e.~word will have exponential mixing. However, to obtain  that for a.e.~realization \emph{all} functions mix, one must take advantage of properties of the function space, such as separability. This is what we do below.

\subsection{Definitions}
The notions of quenched and annealed make sense outside of the IID setting. Although we focus here on the IID case, we will consider these more general notions in Section \ref{sec:annealed_from_quenched}. These definitions make sense in a wider context of skew products. As before, let $M$ be a closed manifold and let $\sigma$
be an automorphism of a probability space $\Omega$ preserving a probability measure $\bmu$. Consider the map $T\colon \Omega\times M\mapsto\Omega\times M$ given by the formula
\begin{equation}
\label{Skew}
 T(\omega, x)=(\sigma \omega, f_\omega (x))  ,
\end{equation}
where for each $\omega$ the map 
$f_\omega\in \Diff^{1}_{\vol}(M).$ Then the iterates of $T$ satisfy
$T^n(\omega, x)=(\sigma^n \omega, F^N_\omega(x))$. The two main concepts of this paper are the following. 

\begin{definition}\label{defn:annealed_exponential_mixing_skew}
We say that the (skew product) random system $T$ in \eqref{Skew}  enjoys {\em annealed exponential mixing} on $\mc{H}^p_0(M)$ if there exist $\alpha,C>0$ such that
for all $A, B\in \mathcal{H}^p_0(M)$ we have
\[
\left| \mathbb{E}_{\bmu} \left(\int A(x) B(F^N_\omega x) dx\right) \right| \leq C e^{-\alpha N} \|A\|_{\mathcal{H}_0^p} \|B\|_{\mathcal{H}_0^p}.
\]
We similarly define this notion for $C^{\alpha}(M)$. 
\end{definition}

\begin{remark}
The annealed mixing considered in Definition \ref{defn:annealed_exponential_mixing_skew} is sometimes
called {\em relative} exponential mixing for the skew product, since we only consider observables that 
do not depend on $\omega$. However, one can show, see e.g. \cite[Corollary 1.2]{dewitt2024surfaces}, that if $\Omega$ is a smooth manifold,
and the base map $\sigma$ is smooth and enjoys exponential mixing then the whole skew product is exponentially mixing
in the sense that for $A, B\in \cH_0^s(\Omega\times M)$
$$ \left|\iint A (B\circ T^N) d\bmu (\omega) d\vol(x) \right|\leq \|A\|_{s} \|B\|_s. $$
\end{remark}

\begin{definition}\label{defn:quenched_exp_mixing_skew}
We say that the random system $T$ in \eqref{Skew} enjoys {\em quenched exponential mixing} on $\mc{H}^p_0(M)$ if there exists $\alpha>0$
and an almost surely finite random variable $C(\omega)$
such that 
for all $A, B\in \mathcal{H}^p_0(M)$
we have
\begin{equation}
\label{EqQEMix}
 \left| \int A(x) B(F^N_\omega x) dx \right| \leq C(\omega) e^{-\alpha N} \|A\|_{\mathcal{H}_0^p} \|B\|_{\mathcal{H}_0^p}.
 \end{equation}   
\end{definition}

Given a function $A$ on $M$ denote 
\begin{equation}\label{defn:S_N}
\DS S_N A(x, \omega)=\sum_{n=0}^{N-1} A(F^n_\omega x).
\end{equation}

\begin{definition}\label{defn:annealed_CLT}
    We say that the random system \eqref{Skew} enjoys the {\em annealed Central Limit Theorem} if there exists $p\geq 0$ and a map $\cD\colon \mathcal{H}^p_0(M)\to \reals$, which is not identically equal to $0$,
     such that for each $A\in \mathcal{H}_0^p$, 
    if $x$ is uniformly distributed on $M$ with respect to volume and $\omega$ is distributed according to $\bmu$ then
    $S_N A(x,\omega)/\sqrt{N}$ converges as $N\to\infty$ to the normal distribution with zero mean and variance $\cD(A).$
    \end{definition}

\begin{definition}\label{defn:quenched_CLT}
   We say that the random system \eqref{Skew} enjoys the {\em quenched Central Limit Theorem} on $\mc{H}^p_0(M)$ if there is a map $\mathsf{D}\colon \mathcal{H}^p_0(M)\to\reals$, which is not identically equal to $0$,
     such that for each $A\in \mathcal{H}_0^p(M)$, there are random variables $a_N(\omega)$ called the quenched drift, and $q_N(\omega)$, called the quenched variance, such that for $\bmu$-almost every 
    $\omega$
    if $x$ is uniformly distributed on $M$ then
     \[
     \frac{S_N A(x, \omega)-a_N(\omega)}{q_N(\omega)}
     \]
     converges as $N\to\infty$ to the normal distribution with zero mean and variance $\mathsf{D}(A).$  
\end{definition}    

\subsection{Statements of main results}
Our main results allow to derive the conclusions of quenched limit theorems from more easily accessible annealed results.
We work in the setting of IID random systems. The following are the main results of this paper.

\begin{theorem}
\label{ThQMix}
Suppose that $\mu$ is a measure supported on $\Diff^r_{\vol}(M)$, $r\ge 1$ and that the associated two-point motion $F_{2}^N$ enjoys annealed exponential mixing on $\mc{H}^p_0(M\times M)$, i.e.~there exists $C\ge 0$ such that for $A,B\in \cH^p_0(M\times M)$,
\begin{equation}
\label{EqAn2Mix}
\left| \EXP\left(\iint A(x, y) B(F_{2}^N (x, y)) dx dy\right)\right| 
\leq C \|A\|_{\cH^p_0} \|B\|_{\cH^p_0} e^{-\alpha N} . 
\end{equation}
Then $F$ satisfies quenched exponential mixing, that is, for all $s>0$
there exists $\beta>0$ such that for almost every $\omega$ there exists $ C=C(\omega)$ such that 
 for
all $A, B\in \cH^s_0$
\begin{equation}
\label{EqQEM}
 \left| \int A(F^{N}_{\omega} x) B(F^{N+k}_{\omega} x) dx\right| \leq C N \|A\|_{\cH^s_0} \|B\|_{\cH^s_0} e^{-\beta k} . 
 \end{equation}
Moreover, there is a polynomial tail on $C(\omega)$: There exists $C'$ such that that $\mathbb{P}(C(\omega)>D)\le C'D^{-2}$.
\end{theorem}

We note that the annealed exponential mixing assumption on $\mc{H}^p_0$ implies that \eqref{EqQEM} holds for $C^{\alpha}$ as well. Further, as annealed exponential mixing on $\mc{H}^p_0(M\times M)$ and $C^{\alpha}(M\times M)$ are equivalent by Remark~\ref{rem:holder_sobolev_equivalence}, \eqref{EqQEM} holds under the assumption of exponential mixing on $C^{\alpha}(M\times M)$, $\alpha>0$.

Next we discuss the quenched CLT. 
For this we will need to describe the variance of the resulting distribution. 
For $A\colon M\to \R$ let
\[
\mc{D}(A)=\EXP\left(\int A^2 dx dy\right)
+2\sum_{k=1}^\infty \EXP\left(\int A(x) A(F^k x) dx dy\right) 
\]
Given a function $B$ on $M\times M$ let 
\begin{equation}
S_B^N(x,y)=\sum_{n=0}^{N-1} B(F^n x, F^ny),
\end{equation}
and let  
\begin{equation}
\label{2DVar}
\bbD(B)=\EXP\left(\int B^2 dx dy\right)
+2\sum_{k=1}^\infty \EXP\left(\int B(x,y) B(F^k x, F^k y) dx dy\right). 
\end{equation}

In the proof of Theorem \ref{ThAnQCLT} below, we will apply this formula where $B$ has either the special form $B(x,y)=A(x)-A(y)$ or $B(x,y)=A(x)$. In these cases a simple calculation using that $\iint A(x)A(F^k(y))\,dx\,dy=0$ yields that
\begin{equation}\label{eqn:Ax-Ay}
\bbD(A(x)-A(y))=2\mathcal{D}(A)\text{ and }\bbD(A(x))=\mathcal{D}(A).
\end{equation}

\begin{theorem}
\label{ThAnQCLT}
Suppose that $\mu$ is a probability measure supported on $\Diff^1_{\vol}(M)$ and that for the associated two-point motion we have annealed exponential mixing on $\mc{H}^p(M\times M)$ for some $p\ge 0$ or on $C^{\alpha}(M\times M)$ for some $\alpha>0$, i.e.\ \eqref{EqAn2Mix} holds.  
Moreover, 
suppose that
for each $C^{\infty}$ function $B(x,y)\in C^{\infty}(M\times M)$, 
we have that
$S^N_B$ satisfies the Central Limit Theorem with polynomial convergence 
of characteristic functions, that is: there exists $\eta>0$ such that for each $\xi\in \R$:
\begin{equation}
\label{Eq2DPolyChar}
 \EXP\left(\iint e^{iS^N_B(x,y) \xi/\sqrt{N}} dx dy\right)=e^{-\bbD(B) \xi^2/2} 
  +O_{B}\left(N^{-\eta}\right) .
\end{equation}
Then for each $s>0$, with probability $1$ for each $A\in \mathcal{H}^s_0(M)$ it holds that 
as $N\to \infty$ that $ N^{-1/2} S_N(A)$ converges to a normal random variable with zero mean and variance $\cD(A)$. 

The quenched CLT also holds with $\mathcal{H}^s_0(M)$ replaced by $C^s_0(M)$---\discretionary{}{}{}the space of zero mean H\"older functions.
\end{theorem}

 We note that the hypotheses of Theorem \ref{ThAnQCLT} is satisfied  for systems 
with spectral gap \cite[Thm.~3.7]{Gouezel}, 
see the discussion in the proof of \cite[Thm.~7.13]{dewitt2025conservative} for  details. 
Hence we obtain: 
\begin{corollary}
    If the generator of the two-point motion of $\mu$ has a spectral gap on $\cH^s_0$ for some $s\in \R$ 
    then quenched exponential mixing holds on $\mathcal{H}^p_0$ with $p>0$, and the quenched 
    Central Limit Theorem holds for $C^r$ functions with $r>0$.
\end{corollary}

\subsection{Structure of the paper} The structure of the paper is the following. In Section \ref{sec:counterexamples}, we give some counterexamples to motivate our main theorems. 
Technical preliminaries are collected in 
Section~\ref{sec:prelims}. Then in Section \ref{sec:annealed_from_quenched} we study when quenched results imply annealed results. In Section \ref{sec:exponential_mixing}, we show that annealed exponential mixing of the two point motion implies quenched exponential mixing of the one point motion.
In Section \ref{sec:var}, we prove some preliminary estimates controlling the convergence of the quenched variance. 
Then in Section \ref{sec:central_limit_theorem}, we show how to deduce the quenched central limit theorem from the annealed one. 
\\

\noindent\textbf{Acknowledgments.} The first author was supported by the National Science Foundation under Award No.\ DMS-2202967. The second author was supported by the National Science Foundation under award No.\ DMS-2246983. 

\section{Counterexamples}\label{sec:counterexamples}
In general, quenched and annealed results are inequivalent even for IID random maps. Our first example shows that even if the annealed dynamics averages perfectly after a single iterate that there might still be no quenched result. In particular, these examples show that the hypotheses on the two-point motion that we consider are quite natural.

\begin{example} (Uniformly Random Translations on $\mathbb{T}^d$) 
Let $\omega_n\in (\mathbb{T}^d)^\mathbb{N}$ be uniformly distributed on $\mathbb{T}^d$ and let $f_\omega(x)=x+\omega_1$.
Then $x_N=F_\omega^N x$ are IID uniformly distributed on $\mathbb{T}^d$ so the system enjoys annealed exponential 
mixing. However, in this case $x_N=x_0+W_N$ where 
\begin{equation}
    \label{TTInvSum}
W_N=\sum_{n=0}^{N-1} \omega_n,
\end{equation}
so letting $A\!=\!e^{i\langle k, x\rangle}$, $B\!=\!e^{-i\langle k, x\rangle}$  for some $k\!\neq\! 0$ 
we see that
$|\int A(x) B(F^N_\omega(x))dx|\!=\! 1$ for all $N.$
Thus the system does not have quenched mixing. 

Next we consider the central limit theorem for this system. Certainly the system satisfies the annealed central limit theorem for the one-point motion. For any $A\colon \mathbb{T}^d\to \R$, the distribution of $S_NA(x,\omega)$, where $x$ is distributed according to $\vol$ and $\omega$ is distributed as above, is the same as the distribution of $\sum_{i=1}^n Y_i$ where the random variables $Y_i$ are IID with the same law as $A_*(\vol)$. Hence the annealed central limit in this case reduces to the classical sum of IID random variables. However, the quenched theorem does not hold. Consider the case where $A(x)=e^{i\langle k,x\rangle}$. Then
\[
S_NA(x,\omega)=\sum_{i=1}^n e^{i\langle k,x+\sum_{j=1}^i \omega_j\rangle}=e^{i\langle k,x\rangle}\sum_{i=1}^n e^{i\langle k,\sum_{j=1}^i\omega_j \rangle}.
\]
But such a sequence of random variables cannot satisfy the central limit theorem because their distribution is just a rescaled version of $e^{i\langle k,x\rangle}$; this fails to ever approach a non-trivial normal distribution.
\end{example}

Next we give two examples where quenched exponential mixing and the quenched central limit theorem hold, but the annealed result fails. 

\begin{example}
\label{ExT-TInv}
 (a)   Let $g$ be a linear Anosov diffeomorphism of $\mathbb{T}^d$. Let $\omega_n$ be IID integer valued random variables 
   where $\mathbb{P}(\omega_n=-k)=\frac{0.001}{k^3}$ for $k<0$ and 
   $ \mathbb{P}(\omega_n=1)=\mathbb{P}(\omega_n=2)=(1-0.001\zeta(3))/2$. 
Let $f_\omega=g^{\omega_0}$. Then $F^N_\omega=g^{W_N}$ where $W_N$ is given by \eqref{TTInvSum}.
 Using the $.001$ factor, it is easy to see that  $\mathbb{E}[\omega_n]\in [1, 2]$, so by the Strong Law of Large Numbers for almost every
$\omega$ we have that
$W_N>N$ for large $N$. Then
$\int A(x) B(g^{W_N} x) dx$ decays exponentially due to the exponential mixing of $g$. 
On the other hand letting $A$ and $B$ be trigonometric functions as in the previous example we see that
$ \int A(x) B(F^{W_N} x) dx=\delta_{W_N, 0}$ whence 
$$ \mathbb{E}\left[\int A(x) B(F^{W_N} x) dx\right]=\Prob(W_N=0). $$
Since 
\[
\mathbb{P}(W_N=0)\geq \sum_{k=N}^{2N} \mathbb{P} (W_{N-1}=k) \mathbb{P}(\omega_{N-1}=-k)\geq C N^{-2},
\] 
the annealed correlations for this system decay only polynomially.

(b) Now define $f_\omega$ as in part (a) but suppose that $\omega_n$ takes values $\pm 1$ with probability 1/2. Then the quenched Central Limit Theorem holds, but the annealed one fails.
 Indeed, note that in this case
\[
S_N A(x, \omega)=\sum_{n=-N}^N \ell(n, N, \omega) A(g^n x),
\]
where $\ell(n, N, \omega)$ is the number of times $K\leq N$ such that $W_k=n.$
Let $V_N(\omega)=\|S_N A(\cdot, \omega)\|_{L^2}^2$ denote the quenched variance.
It follows from \cite{DDKN3} that for a.e.~$\omega$, 
$S_N A(x, \omega)/\sqrt{V_N}$ converges to the normal distribution as $N\to\infty$. 
In fact \cite[Lemma 4.2]{DDKN3} only claims that there is a set $X_N\subset \mathbb{T^d}$ of $\omega$s such that 
$\DS \lim_{N\to\infty} \vol(X_N^c)=0$ ($X_N^c$ denotes the complement of $X_N$) and for $\omega\in X_N$ the distance between the law of $S_N A(\cdot, \omega)$
and the normal distribution converges to zero,  but it does not say that almost every $\omega$ belongs to all $X_N$
with $N\geq N(\omega).$ This weaker statement was sufficient for the purposes of \cite{DDKN3}. However, in fact the proof
of Lemma 4.2 also gives the stronger conclusion described above.
This is because the proof in \cite{DDKN3} proceeds by verifying the conditions of \cite{BG20}, which are abstracted in that paper as \cite[Prop.~2.1]{DDKN3}.
The only condition of that proposition
that was not shown in \cite{DDKN3} to hold almost surely is condition (b). In our setting this condition
amounts to showing that
\begin{equation}
\label{EqFLRW}
\max_n \ell^2(n, N)/V_N\to 0\quad\text{as}\quad N\to\infty.
\end{equation}
To see that this holds almost surely note that it was observed in \cite{DDKN3} that $\max_n \ell(n,N)<N^{0.51}.$
Also with probability 1, $\ell(n,N)=0$ for $|n|\geq N^{0.51}$ for large $N$ since $\DS \max_{n\leq N} |W_n|\leq N^{0.51}$
as follows, for example, from the Azuma's inequality. Finally $\DS \sum_n\ell(n, N)=N$.
Combining the last three facts we easily obtain \eqref{EqFLRW},  and thus we obtain the quenched CLT.

On the other hand, \cite[Proposition 4.1]{DDKN3} also shows that $\sqrt{V_N}/N^{3/4}$ converges in law as $N\to\infty$
to $\|\mathfrak{l}\|_{L^2}$ where $\mathfrak{l}$ is the local times of a Brownian motion started from $0$ 
at time $1$.
It follows that $S_N A/N^{3/4}$ converges in law as $N\to\infty$ to $\fN\times \mathfrak{l}$ where
$\fN$ has normal distribution, $\mathfrak{l}$ is the local time as above and $\fN$ and $\mathfrak{l}$ are independent
(this last result extends \cite{KS79} where a special observable $A$ was considered). In particular, annealed central limit theorem fails as this distribution is not normal.
\end{example}

We note that Examples \ref{ExT-TInv}(a) and (b) are special cases of so called {\em generalized $(T, T^{-1})$ transformations}.
More information on limit theorems for these systems can be found in \cite{dolgopyat2022mixing, DDKN2}. 

 Next, we will give an example that shows that even if annealed exponential mixing holds for the two point motion on all $\mc{H}^s_0$, $s\ge 0$, that there still might not be quenched exponential mixing for the $1$-point motion on $L^2$. 

\begin{example}
Suppose that $\mu$ is a measure on $\SL(2,\mathbb{Z})$ and consider associated random dynamics on $\mathbb{T}^2$. Due to, e.g.~ \cite[Thm.~1.1, Cor.~1.2]{dewitt2025conservative}, if $\mu$ is coexpanding on average forwards and backwards, then the map $\phi\mapsto \int \phi\circ A\,d\mu(A)$ has spectral gap on $L^2(\vol)$. As the $2$-point motion is also coexpanding on average, it follows that the $2$-point motion has spectral gap on $L^2$ as well. In particular the random dynamics of the $1$ and $2$-point motion are annealed exponentially mixing on $\mc{H}^s_0$, for all $s\ge 0$. However, the quenched random dynamics on $L^2$ is an isometry because the dynamics is volume preserving. 
Hence $F_\omega^N$ can not satisfy \eqref{EqQEMix} for
$B=A\circ F_\omega^{-N}$.
\end{example}

\section{Preliminaries}\label{sec:prelims}
In this section we recall some standard facts and introduce notation that will be used below. 

\subsection{Harmonic analysis} We will make extensive use of the Sobolev spaces $\mathcal{H}^p_0$. For many useful facts about these spaces see \cite{shubin2001pseudodifferential} or \cite{lefeuvre2024microlocal}. In particular, in all the arguments below we will work with a fixed 
basis $M$ of $L^2(M)$. These are the eigenfunctions $\varphi_i$ of the Riemannian Laplacian $\Delta$. We fix a basis $\varphi_i$ of eigenfunctions of the Laplacian $\Delta$ that are normalized so that $\|\varphi_i\|_{L^2}=1$. For a zero integral function $A$ write $\DS A=\sum_i a_i \varphi_i$, where each $a_i\in \R$. Then we define the $s$-Sobolev norm by:
\begin{equation}\label{eqn:H_s_defn}
\|A\|_{\mathcal{H}^s_0}^{2}=\sum_{i\in \mathbb{N}} \abs{a_i}^2\lambda_i^{2s}. 
\end{equation}

Below we will use some basic estimates on these functions, such as 
\begin{equation}\label{eqn:C_0_phi_i}
\|\varphi_i\|_{\mathcal{H}^p}=\lambda_i^{p/2} \text{ and }\|\varphi_i\|_{C^0}\le C\lambda_i^{d/2}.
\end{equation}
The first estimate holds since $\varphi_i$ are eigenfunctions of the Laplacian, while the seconds follows from the Sobolev Embedding Theorem (Lemma \ref{LmSobolev}). 

We will  use of the Weyl law for the eigenvalue of the Laplacian. One consequence is the following, which we state as a lemma as we will use it several times.
\begin{lemma}\label{lem:weyl_law_sum}
(Weyl Law)
Suppose that $M$ is a Riemannian manifold of dimension $d$ and $\{\lambda_i\}_{i\in \mathbb{N}}$ are the eigenvalues of the Laplacian. Then there exists $C$ such that the number of eigenvalues of norm at most $\lambda$ is at most $C\lambda^{d/2}$. 
In particular,
\[
\sum_{i} \lambda_i^t 
\]
is finite as long as $t<-d/2$.
\end{lemma}
\begin{proof}
To begin, let $b_n$ the number of eigenvalues of magnitude less than or equal to $n$. Then the sum in question is bounded above by 
$\sum_{n\in\mathbb{N}} n^{-\alpha}(b_n-b_{n-1})$  where $\alpha=-t$. 
Summation by parts shows that
\[
\sum_{n=0}^N n^{-\alpha}(b_n-b_{n-1})=(N+1)^{-\alpha}b_{N+1}-\sum_{n=0}^N((n+1)^{-\alpha}-n^{-\alpha})b_{n+1}.
\]
By the Weyl Law, $b_n\le n^{d/2}$. Thus the sum is convergent as long as $\alpha>d/2$.
\end{proof}
We will also use the usual Sobolev embedding theorem:

\begin{lemma}
\label{LmSobolev}
(Sobolev embedding theorem) Let $M$ be a closed, smooth Riemannian manifold. Then:
$C^s(M)\subset H^s(M)\subset C^{s-d/2}(M)$. 
\end{lemma}
Finally, we record some estimates about the smoothing operators given by projecting onto different parts of the spectrum. For $\lambda\ge 0$ and $\phi\in L^2(M)$, we let $\mc{T}_{\lambda}\phi$ and $\mc{R}_{\lambda}\phi$ to be the projection onto the modes of norm $\le \lambda$ and greater than $\lambda$ respectively. Then $\phi=\mc{T}_{\lambda}\phi+\mc{R}_{\lambda}\phi$. Further, if $\phi\in \mc{H}^s(M)$, $s'\ge s$, then the following estimates come from the definition of the Sobolev norm:
\begin{equation}\label{eqn:smoothing_operator_estimates}
\|\mc{R}_{\lambda}\phi\|_{L^2}^2\le \lambda^{-2s}\|\phi\|_{\mc{H}^s}^2 \text{ and } \|\mc{T}_{\lambda}\phi\|_{\mc{H}^{s'}}^2\le \lambda^{2(s'-s)}\|\phi\|_{\mc{H}^s}^2.
\end{equation}

\begin{remark}
Below, we will keep track of some constants involved in exponential mixing. These constants depend on the specific choice of norm on $\mc{H}^s$. The exponential rate does not depend on this, but the constant $C$ does. Hence it may seem strange to keep track of it. However, as these Sobolev spaces naturally arise from the eigenfunctions of the Laplacian, it seems natural to keep track of constants by using the specific norm on $\mc{H}^s$ defined above as there are natural comparisons between the different Sobolev norms that hold when the norms are defined in this canonical way.
\end{remark}

Next we will record another standard lemma, which will be helpful because the H\"older spaces are not separable.

\begin{lemma}\label{lem:dense_C_alpha}
Suppose that $M$ is a closed Riemannian manifold. Then for any $\alpha>\beta>0$, the inclusion of $C^{\alpha}(M)$ into $C^{\beta}(M)$ is compact. Moreover, in the $C^{\beta}$ topology $C^{\alpha}(M)$ has a countable dense subset. In particular, $C^{\infty}(M)$ functions are dense in $C^{\alpha}(M)$ in the $C^{\beta}$-topology.
\end{lemma}

\noindent This lemma follows from Thm.~A.9 and Thm.~A.10 in \cite{hormander1976boundary}, which estimate how well the mollification of a $C^{\alpha}$ function $u$ approximates $u$ in the $C^{\beta}$ norm.  
Those estimates show that
$C^r$ is dense in $C^\alpha$ for each pair $(\alpha, r)$.
As $C^r$ is separable for $r\in \mathbb{N}$, Lemma \ref{lem:dense_C_alpha} follows.

\subsection{Probability} 
 Throughout the rest of the paper, we will write $\mathbb{P}$ and $\mathbb{E}$ for the probability and expectation of a random variable; when we do this we are exclusively taking expectations over the random dynamics $\omega$. 

We will use  a couple of different concentration inequalities. The first one is Azuma's inequality.

\begin{proposition}\cite[Thm.~1.3.1]{steele1997probability}
\label{prop:azuma}
(Azuma's inequality) Suppose that $X_1,\ldots,X_n$ is a martingale difference sequence. Then 
\[
\mathbb{P}\left(\left|\sum_{i=1}^n X_i\right|\ge \lambda\right)\le 2\exp\left(\frac{-\lambda^2}{2\sum_{i=1}^n \|X_i\|^2_{L^{\infty}}}\right).
\]
\end{proposition}

The second result allows to control the growth of ergodic sums for rapidly mixing systems.

\begin{theorem}\label{thm:gaal-koksma}\cite[App.~1]{philipp1975}
(Gaal--Koksma Strong Law) Let $(X_n)_{n\in \N}$ be a sequence of centered random variables with finite variance. Suppose that there exist constants $C$ and $\sigma>0$ such that for all $m\ge 0$ and $n>0$,
\begin{equation}\label{eqn:gaal-koksma-hypothesis}
\mathbb{E}\left[\left(\sum_{m<j\le n+m}X_j\right)^2\right]\le C((m+n)^{\sigma}-m^{\sigma}),
\end{equation}
Then for each $\delta>0$ almost surely,
\begin{equation}
\sum_{j\le N} X_j=O(N^{\sigma/2}\ln^{\delta+2}(N)). 
\end{equation}
\end{theorem}

\section{Deriving annealed results from the quenched ones}\label{sec:annealed_from_quenched}
Here we recall the basic tools for deducing annealed results from the quenched ones. We will work in the general framework 
of skew products \eqref{Skew}.

\begin{proposition}
Suppose that the skew product \eqref{Skew} satisfies quenched exponential mixing \eqref{EqQEMix} on $\mc{H}^p_0(M)$ for some $\alpha>0$ and that corresponding 
prefactor $C(\omega)$ has a power tail:
$$ \mathbb{P}_{\bmu}(C(\omega)>R)\leq K/R^\kappa, $$
for some $K, \kappa>0$. Then annealed exponential mixing (Def.~\ref{defn:annealed_exponential_mixing_skew}) holds on $\mc{H}^p_0(M)$ with rate $\alpha'$ for any $0<\alpha'\le \alpha\kappa/(1+\kappa)$.
\end{proposition}

\begin{proof}
Note that $\left|\int A(x) B(F^N_\omega x) dx\right|\leq \min\left\{C(\omega) e^{-\alpha N}, 1 \right\}
\|A\|_{\mathcal{H}^p_0} \|B\|_{\mathcal{H}^p_0}$. Let $\beta\ge 0$, then the expectation of the minimum is bounded by 
\begin{equation}\label{eqn:upperbound_12}
e^{(\beta-\alpha)N}+\mathbb{P}_{\bmu}(C(\omega)>e^{\beta N})\le e^{(\beta-\alpha)N}+Ke^{-\kappa\beta N}. 
\end{equation}
The right hand side is dominated by whichever term has the bigger exponent. The choice of $\beta$ that maximizes $\min\{\alpha-\beta,\kappa \beta\}$ is $\beta=\alpha/(1+\kappa)$. In this case the right hand side of \eqref{eqn:upperbound_12} is of order $e^{-\alpha\kappa/(1+\kappa)N}$.
\end{proof}

The next result allows us to obtain the annealed Central Limit Theorem from the quenched one. The following result is basically \cite[Lemma 5.6]{DDKN2} (The proof there also works for $\mc{H}^s$). The proof in the present case is slightly simpler as we do not consider functions $A\colon \Omega\times M\to \R$.

\begin{proposition}
Suppose that the skew product \eqref{Skew} satisfies the quenched Central Limit Theorem (Def.~\ref{defn:quenched_CLT}), 
Fix a function $A\colon M\to \R$ in  $\mc{H}^s(M)$, $s\ge 0$, or $C^{\alpha}(M)$, $\alpha>0$, and
suppose that the quenched
variance satisfies that $q_N(\omega, A)/\sqrt{N}$ converges in law as $N\to\infty$ to a constant $q=q(A)$ while
the quenched drift satisfies that $\frac{a_N(\omega, A)}{\sqrt{N}}$ converges as $N\to\infty$ to a normal distribution with
zero mean and variance $\mathfrak{D}(A)$.
Then the annealed Central Limit Theorem (Def.~\ref{defn:annealed_CLT}) holds, that is, if $\omega$ is distributed according to $\bmu$ and
$x$ is uniformly distributed on $M$, then 
$\frac{S_N A(\omega, x)}{\sqrt{N}}$ converges in law as $N\to\infty$ to a normal random variable with zero mean 
and variance $\cD(A)=\mathsf{D}(A) q(A)+\mathfrak{D}(A).$
\end{proposition}

\section{Exponential Mixing}\label{sec:exponential_mixing}
Here we show that exponential mixing of the two point motion implies quenched exponential mixing of the one-point motion. The idea is that one can study the decay of correlations for a basis of $\cH^p_0$ comprised of eigenfunctions of the Laplacian, and deduce that most words exhibit good decay for all the low modes in this basis.

\begin{proof}[Proof of Theorem \ref{ThQMix}]
What then remains is the case that $p>0$ and we must conclude that we may take $s>0$ in the conclusion. By interpolation (Lemma \ref{lem:interpolation_of_mixing}), it is sufficient to prove the result for $s$ sufficiently large, so we assume in the computations below that
$s>p+(3d/2)$.

Let $\varphi_j$ be an orthonormal basis consisting of eigenfunctions of $\Delta$. Then
$\DS \Delta \varphi_j\!\!=\!\!\lambda_j^2 \varphi_j,$ $\|\varphi_j \|_{L^2}=1$ and $\|\varphi_j\|_{\cH^{p}_0}={\lambda_j^p}$.
Denote
\begin{equation}\label{eqn:defn_rho}
\rho_{i,j, n, k}=\int (\varphi_i\circ F^n)  (\varphi_j\circ F^{n+k}) dx.
\end{equation} 
Then because $\rho_{i,j,n,k}$ and $\rho_{i,j,0,k}$ have the same distribution, 
$$ \EXP[\rho_{i,j, n, k}] =\EXP[\rho_{i,j, 0, k}]=O\left(\lambda_i^{p}  \lambda_j^{p}  e^{-\alpha k}\right), $$
$$ \EXP[\rho_{i,j, n, k}^2] =\EXP[\rho_{i,j, 0, k}^2]=\iint \varphi_i(x) \varphi_i(y) \varphi_j(F^k x)\varphi_j(F^k y) dxdy  
=O\left(\lambda_i^{2p+d}  \lambda_j^{2p+d}  e^{-\alpha k}\right) $$
where we have used $2$-point mixing \eqref{EqAn2Mix} for the function $\psi(x,y)=\varphi_i(x) \varphi_i(y)$, which satisfies
$$\|\psi\|_{\cH^p}\leq \|\psi\|_{C^p}\leq C \|\varphi_i\|_{C^p}^2\leq 
C \|\varphi_i\|_{\cH^{p+d/2}_0}^2=C\left(\lambda_i^{p+d/2}\right)^2.
$$
See, e.g.~\cite[Thm.~A.7]{hormander1976boundary} for the second inequality, 
$\|\varphi_1\varphi_2\|_{C^p}\le C\|\varphi_1\|_{C^p}\|\varphi_2\|_{C^p}$. 

Let $0<\beta<\alpha/2$, then by Chebyshev's inequality, for any $D>0$,
\begin{equation}\label{eqn:rho_chebyshev}
\Prob\left(|\rho_{i,j, n, k}|>nD \lambda_i^{p+t} \lambda_j^{p+t} e^{-\beta k}\right)\leq C D^{-2} n^{-2} 
 \lambda_i^{d-2t} \lambda_j^{d-2t} e^{-(\alpha-2\beta)k}. 
 \end{equation}
 From Lemma \ref{lem:weyl_law_sum},
 $\sum_i \lambda_i^{-\alpha}$ is finite for $\alpha>d/2$. Thus summing over $i,j,n,k$, we see that as long as $t>d$, there exists $C'$ such that for every $D$, we have that 
 \begin{equation}\label{eqn:ijnk}
\Prob\left(|\rho_{i,j, n, k}|>nD \lambda_i^{p+t} \lambda_j^{p+t} e^{-\beta k} 
\text{ for some } i,j,n,k\right)\leq D^{-2}C'.
 \end{equation}
 
 Let us now consider the mixing of a word $\omega$ where there exists $D>0$ such that for all $i,j,n,k$,
 \begin{equation}
 \label{BasisMixing}
|\rho_{i,j, n, k}|\leq Dn \lambda_i^{p+t} \lambda_j^{p+t} e^{-\beta k}.
\end{equation}
Decompose $A=\sum_i a_i \varphi_i,$
$B=\sum_j b_j \varphi_j.$ We will estimate 
$\int (A\circ F^n) (B\circ F^{n+k})\,dx$ by multiplying these series term by term; to do so we first check that things are absolutely summable as long as $s$ is sufficiently large:
$$
\sum_{i,j} \abs{a_i}\abs{b_j}\|\pez{\varphi_i\circ F^n}\pez{\varphi_j\circ F^{n+k}}\|_{L^{\infty}}
\le \left(\sum_i \abs{a_i}\right)\left(\sum_j \abs{b_j}\right)
$$$$
<\left(\sum_i \abs{a_i}^2\lambda^{2s}\right)^{1/2}\left(\sum_i \abs{b_i}^2\lambda^{2s}\right)^{1/2}\left(\sum_i \lambda_i^{-2s}\right),
$$
which is finite as long as $s>d$ by the Weyl Law (Lemma \ref{lem:weyl_law_sum}). Thus by dominated convergence and \eqref{BasisMixing} it follows that
that
\begin{align*}
\left| \int A(F^N x) B(F^{N+k} x) dx \right|&\leq
DN e^{-\beta k} \sum_{i, j} \abs{a_i}\abs{b_j} \lambda_i^{p+t}  \lambda_j^{p+t} \\
&=DN e^{-\beta k} \left[ \sum_{i} \abs{a_i} \lambda_i^{p+t} \right] \left[\sum_j  \abs{b_j} \lambda_j^{p+t}
\right].
\end{align*}
Note that 
\begin{equation*}
\sum_{i} |a_i| \lambda_i^{p+t}\leq \left(\sum_i |a_i|^2 \lambda_i^{2s}\right)^{1/2} \left(\sum_i \lambda_i^{2p+2t-2s}\right)^{1/2}
\leq \|A\|_{\cH^s_0}  \left(\sum_i \lambda_i^{2p+2t-2s}\right)^{1/2}. 
\end{equation*}
As before, by the Weyl Law
the last term is finite for $s>p+t+d/2.$
If this condition  on $s$ holds, then we obtain the same estimate on $\sum_j  \abs{b_j} \lambda_j^{p+t}$ as well.
Note that if $s>p+(3d/2)$, we can choose $t$ so that both $t>d$ and $s>p+t+(d/2)$, so the above estimate holds and we obtain that 
\[
\abs{\int A(F^N x) B(F^{N+k} x) dx}\le C'D Ne^{-\beta k}\|A\|_{\mathcal{H}^s_0}\|B\|_{\mathcal{H}^s_0}
\]
as desired. In particular, from \eqref{eqn:ijnk}, $\mathbb{P}(D>D_0)<C''D^{-2}_0$, so we also obtain the polynomial tail as required.
\end{proof}

The following type of interpolation result is quite standard and works for $C^r$ as well as $\mc{H}^s$ norms; see, for example \cite{siddiqi2023decay}, \cite{bedrossian2022almost}, \cite{cotizelati2020relation}. 
\begin{lemma}\label{lem:interpolation_of_mixing}
(Interpolation of Quenched Exponential Mixing) Suppose that $\omega$ is a sequence of maps in $\text{Homeo}_{\vol}(M)$ such that $f_{\omega}$ exponentially mixes with rate $\alpha$ and constant $C_{\omega}$ on functions in $\mc{H}^{s_0}(M)$, $s_0>0$. Then for each $s>0$, $f_{\omega}$ is exponentially mixing on $\mc{H}^{s}$ with exponential rate $\alpha/(2s_0s^{-1}-1)$ and constant $C_{\omega}+3$. Further, for any $\gamma>0$ $f_{\omega}$ is exponentially mixing on functions in $C^{\gamma}$ with some constants $\beta(\alpha,\gamma)$ and $C'(C_{\omega},\alpha,\gamma)$ that we do not compute explicitly.
\end{lemma}

 \begin{proof}
We begin with the case of functions in $\mc{H}^s$. The case of functions in $C^{\alpha}$ is explained in \cite[Lem.~2.1]{siddiqi2023decay}, and we omit it.

Suppose that $\phi,\psi\in \mc{H}^{s}$ are two zero mean functions and $0<s<s_0$, the nontrivial case. Then we have the following estimates using the smoothing operators $\mc{T}_{\lambda},\mc{R}_{\lambda}$ and the estimates in \eqref{eqn:smoothing_operator_estimates}: 
\begin{align}
\abs{\int \phi\psi\circ f^n_{\omega}\,d\vol}&\le\abs{ \int (\mc{R}_{\lambda}\phi+\mc{T}_{\lambda}\phi)(\mc{R}_{\lambda}\psi+\mc{T}_{\lambda}\psi)\circ f^n_{\omega}\,d\vol}\\
&\le \|\mc{R}_{\lambda}\phi\|_{L^2}\|\mc{R}_{\lambda}\psi\|_{L^2}+\|\mc{R}_{\lambda}\phi\|_{L^2}\|\psi\|_{L^2}+\|\phi\|_{L^2}\|\mc{R}_{\lambda}\psi\|_{L^2}\\
&\hspace{1cm}+C_{\omega}e^{-n\alpha}\|\mc{T}_{\lambda}\phi\|_{\mc{H}^{s_0}}\|\mc{T}_{\lambda}\psi\|_{\mc{H}^{s_0}}\\
&\le (3\lambda^{-s}+e^{-n\alpha} \lambda^{2(s_0-s)}C_{\omega})\|\phi\|_{\mc{H}^s}\|\psi\|_{\mc{H}^s}.
\end{align}
Then take $\lambda=e^{\beta n}$. Optimizing in $\beta$ then yields $\beta=\alpha/(2s_0-s)$.
Thus we get exponential mixing at rate
$\displaystyle
\frac{\alpha}{2(s_0/s)-1}$ with constant $C_{\omega}+3.$
\end{proof}

 An immediate corollary of Lemma~\ref{lem:interpolation_of_mixing} is the following.
\begin{corollary}\label{cor:annealed_interpolation}
Suppose that annealed exponential mixing holds on $\mc{H}^p_0(M)$ for some $p>0$. Then annealed exponential mixing holds on $\mc{H}^s_0(M)$ for all $s>0$. Suppose that annealed exponential mixing holds on $C^r(M)$ for some $r>0$. Then annealed exponential mixing holds on $C^{\alpha}$ for all $\alpha>0$.
\end{corollary}

\begin{remark}\label{rem:holder_sobolev_equivalence}
The arguments in this section show that annealed exponential mixing on $\mc{H}^s(M)$ or $C^{\alpha}(M)$ for any particular positive $s$ or $\alpha$, imply annealed exponential mixing for all $\alpha,s>0$. That $\mc{H}^{s_0}$ implies exponential mixing for all $\mc{H}^s$, $s>0$, is the content of Lemma \ref{lem:interpolation_of_mixing}. That lemma shows that annealed mixing on $\mc{H}^{s_0}$ implies annealed exponential mixing on $C^{\alpha}$. But for sufficiently large $s_1$, $\mc{H}^{s_1}$ embeds compactly in $C^{\alpha}$, hence as this satisfies annealed exponential mixing, we also have annealed exponential mixing on $\mc{H}^{s_1}$. Thus annealed exponential mixing on $\mc{H}^s$ and $C^{\alpha}$, $s,\alpha>0$, are the same thing.
\end{remark}

Our mixing estimate implies an almost sure bound on ergodic sums.
Note that if $A$ is fixed first then Corollary \ref{cor:pointwise_estimate} can be obtained from annealed mixing and the Fubini Theorem. The novelty of this
result
is that the set of $\omega$s of full measure could be taken independent of $A$.

\begin{corollary}\label{cor:pointwise_estimate}
Under the assumptions of Theorem \ref{ThQMix}, for almost every $\omega$ we have that for each 
$A\in \cH^s_0$  with $s>0$ or $A\in C^{\alpha}(M)$, $\alpha>0$, and for almost all $x$:
\begin{equation}\label{eqn:as_birkhoff_est}
 \sum_{n=1}^N A(F^n_{\omega} x)=O\left(N^{1/2+\eps}\right).
 \end{equation}
\end{corollary}

\begin{proof}
This follows from quenched exponential mixing \eqref{EqQEM} and Theorem \ref{thm:gaal-koksma}. Let $X_j=A(F^j_{\omega}(x))$. Then we must check that \eqref{eqn:gaal-koksma-hypothesis} is satisfied. From \eqref{EqQEM}, it follows that:
\begin{align*}
\abs{\left(\sum_{j=M+1}^{N+M} A(F^j_{\omega}x) \right)^2}&\le \sum_{0<i,j\le N} \abs{\int A(F^{M+i}_{\omega}x)A(F^{M+j}_{\omega}x)\,dx}\\
&\le \sum_{0<i-j\le N}\min\{1,C(\omega)e^{-\beta \abs{i-j}}(M+\min\{i,j\})\}. 
\end{align*}
From this it follows that \eqref{EqQEM} holds with $\sigma=1+\epsilon'$ for any $\epsilon'>0$. Hence by Theorem \ref{thm:gaal-koksma}, we obtain \eqref{eqn:as_birkhoff_est}.
\end{proof}

\section{Asymptotic Quenched Variance from exponential mixing} 
\label{sec:var}

\subsection{Main result.}
To show the quenched Central Limit Theorem from the annealed Central Limit Theorem 
we need to control the growth of quenched variances. 
In this section, we study the quenched variance both identifying the growth rate of the sample variance as well as showing that the quenched variance is almost surely independent of the sequence. The rest of this section is structured as follows. First, we will give some definitions separating the terms appearing in the variance into those where we do not use mixing and those where we do. Then we discuss some properties of the function spaces that are used in the arguments that do not use mixing (Def.~\ref{defn:SFOP}) and obtain concentration estimates (Lemma \ref{claim:low_order_terms}). After that we prove estimates for the terms where we do use mixing (Lemma \ref{claim:control_of_high_order_variance}). These combine to give a quenched estimate on the convergence of the sample variance (Proposition~\ref{prop:quenched_variance}).

The the main result of the section is the uniform control of the quenched variance. It plays a key role in the
proof of quenched CLT (Theorem \ref{ThAnQCLT}) given in Section \ref{sec:central_limit_theorem}.

\begin{proposition}\label{prop:quenched_variance}
Suppose that $\mu$ is a compactly supported measure on $\Diff^1_{\vol}(M)$ that satisfies quenched exponential mixing with polynomial tails on functions in $C^{\alpha}(M)$ or $H^s(M)$, $s,\alpha>0$, or more generally a function space $\mc{B}$ 
 of controlled $L^2$ complexity in the sense of 
Definition~\ref{defn:SFOP}. Then quenched variance converges sublinearly. Namely, for $\mu^{\N}$-a.e.~$\omega$,
\begin{equation}
\abs{V_N(A)-\cD^2(A)}=o(N)\|A\|^2_{\mc{B}}. 
\end{equation}
\end{proposition}

\subsection{Good Banach spaces.}
In order to study the convergence of the variance there are two regimes of behavior. To control the variance in the regime where there is not enough time for mixing, we will rely on Azuma's inequality as well as properties of function spaces
we work with. 

We will consider a Banach space $\mc{B}\subseteq L^2(\vol)$ satisfying a list of properties. If you do not like looking at the list, just pretend that $\mc{B}$ is equal to the space of H\"older functions and nothing will be lost. The main statement is that from the point of view of $L^2$, one can find an $\epsilon$-dense set in the unit sphere of $\mc{B}$ at a particular quantitative rate. Here is the list of properties that we require.

\begin{definition}\label{defn:SFOP}
A Banach space $\mc{B}\subseteq L^2(\vol)$ has \emph{controlled $L^2$ complexity} if
\begin{enumerate}
    \item (Controlled by $L^2$) There exists $C_1$ such that $\|f\|_{\mc{B}}\le 1$ implies $\|A\|_{L^2}<C_1$. 
    \item\label{item:L2_denseness}
    (Denseness with respect to $L^2$ ) We say that a subset $\mc{F}$ of $\mc{B}$ is $\epsilon$-dense with respect to $L^2$ 
    if for any $A$ with $\|A\|_{\mc{B}}=1$ there exists $\hat{A}\in \mc{F}$ such that
    \begin{equation}\label{eqn:L2_dense2}
    \|A-\hat{A}\|_{L^2}\le \epsilon(n). 
    \end{equation}
    
     We then require that there exists a sequence $\mc{F}_N$ of finite subsets of $\mc{B}$ of $\epsilon(N)= o(1/\ln(N))$ dense with respect to $L^2$ such that for some $\delta>0$,
    \begin{equation}\label{eqn:SFOP_summability}
\sum_{N\in \N} \abs{\mc{F}_N}e^{-N^{\delta}/C}<\infty.
    \end{equation}
\end{enumerate}
\end{definition}

\noindent Note that $L^2$ itself does not satisfy these properties: unlike the other spaces the inclusion of $L^2$ into itself is not compact.

\begin{remark}\label{rem:bounded_L2_norm}
Note that in Definition \ref{defn:SFOP} we may assume that for all $N\in \N$ and $f\in \mc{F}_N$ we have $\|f\|_{L^2}\le D$
with $D=2C_1$, where $C_1$ is the constant in Condition (1) of the definition. Indeed, we can
just drop all the functions which do not
satisfy this property and Condition (2) would still hold.
\end{remark}

\begin{lemma}
Let $M$ be a closed, smooth Riemannian manifold. Then for  $\alpha,s>0$, the space of $C^{\alpha}$ or $\mc{H}^s$ functions satisfy Definition \ref{defn:SFOP}.
\end{lemma}

\begin{proof}
There are two cases $C^\alpha(M)$ and $H^s(M)$, which we handle separately. Let $\epsilon(n)$ be the smallest $\epsilon$-denseness in $\mc{B}$ with respect to $L^2$ that we can achieve using a set $\mc{F}$ of $n$ functions.
In each case it suffices to calculate, given an $\epsilon$, how many functions suffice to achieve $\epsilon$-denseness with respect to $L^2$. This number is certainly finite as the inclusion of both of these spaces into $L^2$ is compact. As long as the number of functions does not grow too fast as $\epsilon\to 0$, we may conclude. 

\noindent\textbf{Case 1:} $C^{\alpha}(M)$, $\alpha>0$. We would like to be $\epsilon$-dense with respect to $L^2$. Cover the manifold with boxes of diameter of order $(\epsilon/2)^{1/\alpha}$. Note that in each of these boxes the value of any 
$\alpha$--H\"older function with norm $1$ cannot change by more than $\epsilon/2$. 
 Now number the boxes $B_0, \dots , B_\ell$ so that for $j>0$, $B_j$ is adjacent to one of the previous boxes.
Note that $\ell$ is order $\vol(M)/(\epsilon^{d/\alpha})$. Note that the value of any $\alpha$--H\"older norm $1$ function can change by at most $\epsilon/2$ on a box. Let $\Omega_{\epsilon}$ denote the set of centers of these boxes. 
Choose an initial value for the function on $B_0$ in increments of $\epsilon/2$ from $-1$ to $1$. Then we extend the domain of the candidate to the next box by modifying the value of the candidate function on the previous box by an increment of either 
$0, \pm \epsilon/2,\pm \epsilon$, or $\pm 3\epsilon/2$. At the end of this process we produce at most
\begin{equation}\label{eqn:number_of_functions_for_dense_holder}
3 \epsilon^{-1}{\mathsf{K} }^{\vol(M)/\epsilon^{d/\alpha}}
\end{equation}
candidate functions each encoded by a choice of initial value on $B_0$, and then an element in 
$\{1/2,-1/2,1,-1,3/2,-3/2,0\}^{\ell}$ where $\ell=O(\vol(M)/\epsilon^{d/\alpha})$. Of course, not all of these will extend to a H\"older continuous function that is $L^2$ close to a H\"older function of norm $1$. 

 Using these candidate functions, we will now obtain our actual $\epsilon$-dense subset. 
We first claim that any $(1,\alpha)$-H\"older function, when restricted to $\Omega_{\epsilon}$, is $\epsilon$-close to a candidate (with respect to the supremum norm). For any such function $\phi\in C^{\alpha}(M)$, there will exist a candidate that is within $\epsilon$ of the value that $\phi$ takes on $\Omega_{\epsilon}$. This is because as we traverse the set of boxes, we allow ourselves to vary the value on adjacent boxes in increments of $\epsilon/2$. Hence, as $\phi$ cannot change by a value of more than $\epsilon$ between two adjacent boxes, and the candidate can vary by increments of $\epsilon/2$, there is necessarily such a candidate.

Now, discard any candidate function $\phi$ that is not $\epsilon$-close to a $(1,\alpha)$-H\"older function restricted to $\Omega_{\epsilon}$. For the candidates $\phi$ that remain, choose such an $\epsilon$-close H\"older function $\widehat{\phi}\in C^{\alpha}(M)$. We claim that these functions form a $2\epsilon$-dense set. To see this first, note that for any $(1,\alpha)$-H\"older function $\psi$, by the previous paragraph, there exists a candidate function $\phi\colon \Omega_{\epsilon}\to \R$ as well as a $(1,\alpha)$-H\"older function $\widehat{\phi}$ that is $\epsilon$-close to $\phi$ on $\Omega_{\epsilon}$. But this means that $\psi$ and $\widehat{\phi}$ are $(1,\alpha)$-H\"older functions that are $2\epsilon$-close on $\Omega_{\epsilon}$. Due to H\"older-ness, as each of $\psi$ and $\widehat{\phi}$ can vary by at most $\epsilon/2$ on any box $B_i$, and they are within $2\epsilon$ at the center of the box, it follows that $\abs{\psi-\widehat{\phi}}\le 3\epsilon$. Hence $\|\psi-\widehat{\phi}\|_{L^2}\le 3\epsilon$. 

Now we can check that the conditions of Definition \ref{defn:SFOP} are satisfied. Let $\epsilon(N)=(1/\ln N)^2$. Then we can take $\abs{\mc{F}_N}$ to be order 
\[
\frac{1}{(\ln N)^2} e^{(\ln \mathsf{K})(\vol M)  (\ln N)^{\frac{2d}{\alpha}}},
\]
which is summable when multiplied by $e^{-N^{\delta}}$. Thus \eqref{eqn:SFOP_summability} holds in this case, and we may conclude that $C^{\alpha}(M)$ satisfies Definition \ref{defn:SFOP}.

\noindent\textbf{Case 2:} Hilbertian Sobolev spaces $\mc{H}^s(M)$, $s>0$. In this case the argument will be a bit different because $\mc{H}^s$ functions are localized in frequency but not in space but is simpler because the norm is already defined in terms of $L^2$. 

To begin, note that by the definition of the $\mc{H}^s$ norm, if $\phi$ is a function with $\|\phi\|_{\mc{H}^s}^2=1$, then $\sum_{\lambda \ge \lambda_0} \|\phi_{\lambda}\|^2\lambda^{2s}\le 1$. Thus the total $L^2$ mass of $\phi$ carried by modes of order greater than $(\epsilon/2)^{-1/2s}$ must be size at most $\epsilon/2$. Thus we can disregard the modes of order greater than $(\epsilon/2)^{-1/2s}$. Due to the Weyl Law, there are order $\lambda^{d/2}$ modes of less than $\lambda$. Note that abstractly the functions in $\mc{H}^s$ with no modes greater than $(\epsilon/2)^{-1/2s}$ is isomorphic to $L^2$ of a counting measure on $(\epsilon/2)^{-1/2s}$ points: there are just additional weights. In particular, the map from the standard copy of $L^2$ on $(\epsilon/2)^{-1/2s}$ to our weighted copy is a contraction. Thus the number of points needed to form an $\epsilon$-net is bounded above by the number of points needed to form an $\epsilon$-net in the usual unit sphere in $\R^{(\epsilon/2)^{-1/2s}}$. By \cite[Lem.~5.2]{vershynin2012introduction} finding an $\epsilon$-net in $S^{n-1}$ requires at most $(1+2/\epsilon)^n$ points. Thus to be $L^2$ $\epsilon$-dense in the unit ball of $\mc{H}^s$, it suffices to use $e^{\epsilon^{-d/s}\ln(1+2/\epsilon)}$ functions. 

To satisfy Definition \ref{defn:SFOP}, let $\epsilon(N)=1/(\ln N)^2$. Applying the conclusion of the previous paragraph, we see that this requires at most $e^{\ln(N)^{3d/s}}$ functions as long as $N$ is sufficiently large. In particular when we multiply by $e^{-N^\delta}$ the result is summable so \ref{eqn:SFOP_summability} holds, and hence Definition \ref{defn:SFOP} is satisfied by $\mc{H}^s(M)$.
\end{proof}

\subsection{Consequences of mixing.}
We record below several straightforward consequences of exponential mixing which will be used below.

\begin{lemma}\label{claim:mean_tail_ge_gamma}
Suppose that we have annealed exponential mixing for the $1$-point motion on a Banach space $\mc{B}\subseteq L^2(\vol)$ for a measure $\mu$ on $\Diff^1_{\vol}(M)$ with rate $\lambda$. Then for any $\eta>0$, there exists $\gamma>0$ and $D_1>0$ such that 
\[
\abs{\sum_{\ell>\gamma \ln N} \E\left[\int AA\circ f^\ell\,d\vol\right]}\le D_1N^{-\eta}\|A\|^2_{\mc{B}}.
\]
\end{lemma}
\begin{proof}
The proof follows because by annealed exponential mixing:
\[
\abs{\E\left[\int AA\circ f^\ell\,d\vol\right]}\le e^{-\lambda k}\|A\|^2_{\mc{B}}. 
\]
The conclusion then follows by summing $\DS \sum_{\ell\ge \gamma \ln N} e^{-\ell \lambda}$. 
\end{proof}

We now record an additional estimate on $\mathcal{D}(A)$ that will be useful later.

\begin{lemma}\label{lem:mcD_est}
Suppose that there is annealed exponential mixing on $\mathcal{H}^p_0(M)$, $p\ge 0$.
Define 
\begin{equation}
\mathcal{D}(A,B)\coloneqq \int AB\,dx+2\sum_{k=1}^{\infty}\mathbb{E}\left[\int AB\circ F^k\,dx\right],
\end{equation}
so that $\mathcal{D}(A,A)=\mathcal{D}(A)$. Then for any $s>0$, and any $A,B\in \mathcal{H}^p_0(M)$, it follows that:
\begin{equation}
\abs{\mathcal{D}(A)-\mathcal{D}(B)}=\abs{\mathcal{D}(A-B,A+B)}\le C_{p,s}\|A-B\|_{\mathcal{H}^s_0}\|A+B\|_{\mathcal{H}^s_0}. 
\end{equation}
The same holds, \emph{mutatis mutandis} if we have annealed exponential mixing on $C^{\alpha}(M)$ or $L^2(M)$.  
\end{lemma}
\begin{proof}
The lemma follows from two estimates. First, by exponential mixing in \eqref{EqAn2Mix},
$\mathcal{D}(A)-\mathcal{D}(B)=\mathcal{D}(A-B,B+A)$
since the series defining these quantities are absolutely convergent and hence can be rearranged. The second estimate is
$ 
\abs{\mathcal{D}(A,B)}\le C\|A\|_{\mathcal{H}^p_0}\|B\|_{\mathcal{H}^p_0}.
$
It holds because the norm of the $k$th term in the definition of $\mathcal{D}$ is $\|A\|_{\mathcal{H}^p_0}\|B\|_{\mathcal{H}^p_0}e^{-\alpha k}$ from annealed exponential mixing \eqref{EqAn2Mix}. 
\end{proof}

\subsection{Splitting the variance.}
Let 
\begin{equation}
S_N(A)(x)\coloneqq \sum_{n=0}^{N-1} A(F^n x),
\end{equation}
and 
\begin{equation}\label{eqn:cor_def}
\Cor_N(A,B)\coloneqq \int S_N(A)S_N(B)\,dx\hspace{1cm} V_N(A)\coloneqq \Cor_N(A,A).
\end{equation}

We will now show that quenched exponential mixing plus a polynomial tail on the constant as in \eqref{ThQMix} gives control on the quenched variance $V_N(A)$ defined above. When we study $V_{N}(A)$, there are order $N^2$ terms that arise in the definition. Those terms fit into two groups depending on the  value of $\ell$ in the expression 
$(A\circ F^i)( B\circ F^{i+\ell})$. First we will give a lemma showing that terms with $\ell$ small relative to $n$ are highly concentrated around their mean. To this end we define
\begin{equation}
D_{N,\ell,i}=\sum_{0\le k< N/\ell} \int A\circ F^{k\ell}B\circ F^{k\ell +i}\,d\vol.  \hspace{1cm} D_{N,\ell}=\sum_{0\le i< \ell} D_{N,\ell,i}. 
\end{equation}
We will further divide the terms into two regimes: one where we can rely on mixing, and the other where we will rely on a generic concentration estimate that does not use any mixing. 
\begin{equation}\label{eqn:gamma_def}
D^{\le \gamma}_{N,\ell} =\sum_{\ell \le \gamma \ln N}D_{N,\ell},\hspace{1cm} D^{\ge \gamma}_{N,\ell}=\sum_{\ell>\gamma \ln N} D_{N,\ell}. 
\end{equation}
Analogously with \eqref{eqn:cor_def}, we also define $D_{N,\ell}(A,B)$, $D_{N,\ell,i}(A,B)$, etc.

\subsection{Near diagonal terms.}
\begin{lemma}
    \label{claim:low_order_terms}
Suppose that $\mu$ is a compactly supported measure on $\Homeo_{\vol}(M)$. Suppose that $\mc{B}$ is a Banach space of functions $M\to \R$ satisfying Definition \ref{defn:SFOP}, such as $C^{\alpha}$ or $\mc{H}^s$, that is preserved by the action of the elements in the support of $\mu$. For any fixed $\gamma,\delta>0$, the following holds. For $\mu^{\N}$-a.e.~$\omega$ and all $A\in \mc{B}$,
\begin{equation}\label{eqn:Dnli_est}
\abs{D_{N,\ell,i}(A)-\frac{N}{\ell}\E\left[\int AA\circ f^\ell \,d\mu \right]}=o_{\omega}(N/(\ell\ln N))\|A\|^2_{\mc{B}}. 
\end{equation}
Furthermore, 
\begin{equation}\label{eqn:summed_est_le_gamma}
\abs{D^{\le \gamma}_{N}(A)-\mc{D}^2(A) }=o(N\|A\|^2_{\mc{B}}). 
\end{equation}
\end{lemma}

\begin{proof}
To begin, note that by rescaling it suffices to establish the claim for $A$ of norm $1$.

To begin for each $N$, pick a family of functions $\mc{F}_N\subset \mc{B}$ as in Definition \ref{defn:SFOP}. By Remark \ref{rem:bounded_L2_norm}, we may assume that these functions have $L^2$ norm uniformly bounded by some constant $D$.  For a function $A\in \mc{B}$, let $P(A,N,\ell,i)$ be denote the probability below:
\begin{align}
P(A,N,\ell,i)\coloneqq 
\mathbb{P}\left(\abs{D_{N,\ell,i}(A)-\frac{N}{\ell}\E\left[\int AA\circ f^\ell \,d\vol \right]}\ge N^{1/2+\delta} \right).&
\label{eqn:bad_event_ilB}
\end{align}
Then applying Azuma's inequality (\ref{prop:azuma}) implies that for all $A$ such that $\|A\|_{\mc{B}}\le 2$:
\[
P(A,N,\ell,i)\le 2\exp\left(-N^{2\delta}/2C\right).
\]
Thus as long as 
\begin{equation}
\sum_{N\ge 0} \abs{\mc{F}_N}\exp(-N^{2\delta}/2C)<\infty,
\end{equation}
only finitely many of the events in equation \eqref{eqn:bad_event_ilB} occur; this is assured by item \eqref{item:L2_denseness} in Definition \ref{defn:SFOP}. 

We now need to extend this to cover all functions; not just those in $\mc{F}_N$. Now consider an arbitrary $B\in \mc{B}$ with $\|B\|_{\mc{B}}=1$. By Item \eqref{item:L2_denseness}, equation \eqref{eqn:L2_dense2}, for each $N$ there exists $\widehat{B}=\widehat{B}_N\in \mc{F}_N$ such that $\|B-\widehat{B}\|_{L^2}<\epsilon(N)=o(1/\ln N)$. Hence
\[
    \int A(B-\widehat{B})\circ f^k\,d\vol =  o(1/\ln N)\|A\|_{\mc{B}}.
\]
In particular, note that 
\begin{align*}
D_{N,\ell,i}(B)\!=\!
D_{N,\ell,i}(\wh{B}+(B-\wh{B}))\!&\!=\!
D_{N,\ell,i}(\wh{B})\!+\!2D_{N,\ell,i}(\wh{B},(B-\wh{B}))\!+\!\!D_{N,\ell,i}(B-\wh{B})\\&=I_D+II_D+III_D.
\end{align*}
Similarly, let \(E_k(B)=N\ell^{-1}\E\left[\int BB \circ f^k\,d\vol\right]\) and define $E_k(A,B)$ analogously. Then we can expand
\begin{align}
E_k(B)=E_k(\wh{B})+E_k(\wh{B},(B-\wh{B}))+E_k(B-\wh{B})=I_B+II_B+III_B.
\end{align}

We now compare $D_{N,\ell,i}(B)$ with $\ell^{-1}N\E\left[\int BB\circ f^k\right]=E_k(B)$. We expect $I_D$ to be close $I_B$, while the $II_{*}$ and $III_{*}$, $*\in \{B,D\}$, terms will be small enough to disregard. 

First we consider the $I_B$ and $I_D$ terms. As only finitely many of the events in \eqref{eqn:bad_event_ilB} occur for the functions in $\mc{F}_N$, we see that there exists a constant $C(\omega)$ not depending on $\ell,i,N$, such that 
\begin{equation}
\abs{I_D-I_B}=\abs{D_{N,\ell,i}(\wh{B})-\ell^{-1}N\E\left[\int \wh{B}\wh{B}\circ f^k\right]}\le C(\omega)N^{1/2+\delta}. 
\end{equation}

Next, the $II_*$ and $III_*$ terms are small because of equation \eqref{eqn:L2_dense2}: Each summand defining these terms has size $o(1/\ln N)$ due to \eqref{eqn:L2_dense2} and the uniform bound on 
$\|\widehat{B}\|_{L^2}$, and there are at most $N/\ell$ summands in each. So, their total size is 
\[
\abs{II_B}+\abs{III_B}+\abs{II_D}+\abs{III_D}=o(N/(\ell \ln N)). 
\]
Combining these estimates, we see that for all $A$ such that $\|A\|_{\mc{B}}=1$,
\begin{equation}
\abs{D_{N,\ell,i}(A)-\frac{N}{\ell}\E\left[\int AA\circ f^\ell \,d\vol \right]}=o(N/(\ell \ln N))\|A\|^2_{\mc{B}}. 
\end{equation}
By rescaling, the above line then holds for all $A\in \mc{B}$, so we have shown \eqref{eqn:Dnli_est}.

Then equation \eqref{eqn:summed_est_le_gamma} follows by applying Lemma
\ref{claim:mean_tail_ge_gamma} as the sum of the omitted terms are a (deterministic) sublinear contribution. 
\end{proof}

\subsection{Off diagonal terms.}
Next we will prove a claim concerning the terms where $\gamma$ in \eqref{eqn:gamma_def} is large enough that mixing is actually happening. 

\begin{lemma}\label{claim:control_of_high_order_variance}
Suppose that $\mu$ is a measure on $\Diff^1_{\vol}(M)$ satisfying quenched exponential mixing with a polynomial tail for functions in a Banach space $\mc{B}$, i.e.\ there is a polynomial tail on the constant $C(\omega)$ controlling quenched exponential mixing: $\mathbb{P}(C(\omega)>D)\le D^{\beta}$ for some $\beta>0$, where $C(\omega)$ is defined as in \eqref{EqQEM}. 
Then, for any $\eta>0$, there exists $\gamma>0$ such that  for $\mu^{\N}$-a.e.~$\omega$ there exists $C'(\omega)$ such that for all $N$ and all $A\in \mc{B}$, 
\[
\abs{D_{N}^{\ge \gamma}(A)}\le C'(\omega)N^{-\eta}\|A\|^2_{\mc{B}}. 
\]
\end{lemma}

In the proof we shall use the following estimate.

\begin{claim}\label{claim:mini_variance_claim}
Let $C_k=\min\{C(\omega)e^{-\lambda k}, 1\}$. Then there exists $D_1>0$ such that
\[
\max\{\E[C_k],\Var(C_k)\}\le D_1e^{-\lambda k\beta/2}. 
\]
\end{claim}

\begin{proof}
Consider the set $\Omega_{1,k}$ of points $\omega$ with $C(\omega)\le e^{\lambda k /2}$. Then due to the polynomial tail, the complement of $\Omega_{1,k}$, $\Omega_{1,k}^c$, has mass $e^{-\lambda k\beta/2}$. On the set $\Omega_{1,k}$, $C_k\le e^{-\lambda k/2}$, while on $\Omega_{1,k}^c$, $C_k(\omega)\le 1$. The conclusion follows. 
\end{proof}

\begin{proof}[Proof of Lemma \ref{claim:control_of_high_order_variance}]
Observe that
\begin{equation}\label{eqn:c_l_controls_mixing}
\abs{D_{N,\ell,i}(A)}\le \left(\sum_{j=0}^{N/\ell} \min \{C(\sigma^{\ell j+i}(\omega))e^{-\lambda\ell},1\}\right)\|A\|^2_{\mc{B}}.  
\end{equation}
As the terms $\min \{C(\sigma^{\ell j+i}(\omega))e^{-\lambda \ell},1\}$ are independent, with variance at most $D_1e^{-\lambda k \beta/2}$ by Claim \ref{claim:mini_variance_claim}, the parenthetical term has variance of order 
$\DS 
D_1\frac{N}{\ell}e^{-\lambda k\beta/2}. 
$
Thus by Chebyshev,
\[
\mathbb{P}\left(\abs{\left(\sum_{j=0}^{N/\ell} \min \{C(\sigma^{\ell j+i}(\omega))e^{-\lambda\ell},1\}-\frac{N}{\ell} \E[C_{\ell}] \right)}>A\right)\le D_1\frac{\frac{N}{\ell}e^{-\lambda \ell \beta/2}}{A^2}.
\]
Now take $A_{N,\ell}=e^{-\ell\epsilon_1\beta}$. Then 
\begin{equation}\label{eqn:A_N_ell_bound1}
\mathbb{P}\left(\abs{\left(\sum_{j=0}^{N/\ell} \min \{C(\sigma^{\ell j+i}(\omega))e^{-\lambda \ell},1\}-\frac{N}{\ell} \E[C_{\ell}] \right)}\!>\!A_{N,\ell}\right)\!\!\le\!\! D_1N^{}e^{-\lambda \ell\beta/2+2\epsilon_1\ell}\ell^{-1}.
\end{equation}
As long as $\gamma$ is sufficiently large and $\epsilon_1$ is sufficiently small, this is summable over $N\ge 0$, $\ell\ge \gamma \ln N$, and $0\le i\le \ell$. Thus almost surely only finitely many of the events on the left hand side of \eqref{eqn:A_N_ell_bound1} occur. By the estimate on $\E[C_{\ell}]$ in Claim \ref{claim:mini_variance_claim}, as long as $\epsilon_1$ is sufficiently small, it follows from \eqref{eqn:c_l_controls_mixing} that there exists $D_2$ such that 
\begin{align*}
\|A\|_{\mc{B}}^{-2}\abs{\sum_{\ell \ge \gamma \ln N}\sum_{0\le i\le \ell} D_{N,\ell,i}(A)}\le& \!\!\!\!\sum_{\ell\ge \gamma \ln N}\sum_{i\le \ell} A_{N,\ell}+N\ell^{-1}\E[C_{\ell}]\\
\le& D_2\!\!\!\!\sum_{\ell\ge \gamma \ln N}\sum_{i\le \ell} e^{-\ell\epsilon_1\beta}+N\ell^{-1}e^{-\lambda \ell\beta/2}
\le D_3N^{-\eta}
\end{align*}
for some $\eta>0$ as long as $\gamma$ is sufficiently large and $\epsilon_1$ is sufficiently small. Note that if we increase $\gamma$, we can make $\eta$ as small as we like. 

In particular, it follows that
\[
\abs{D_{N}^{\ge \gamma}(A)}\le D_3N^{-\eta}\|A\|^2_{\mc{B}}.
\]
Thus as long as we choose $\gamma$ sufficiently large, the conclusion follows. 
\end{proof}

\subsection{Proof of Proposition \ref{prop:quenched_variance}.}
This is a straightforward conclusion from the estimates proven above. Recall that we fixed $\gamma>0$ and split
\[
V_N(A)=D_N^{\le \gamma}(A)+D_N^{\ge \gamma}(A).
\]
By Lemma \ref{claim:low_order_terms}, for all $\gamma>0$, 
a.e.~$\omega$ and all $A\in \mc{B}$, $D_N^{\le \gamma}(A)=N\mc{D}^2(A)+o(B)\|A\|^2_{\mc{B}}$. By Lemma~ \ref{claim:control_of_high_order_variance},  for a.e.~$\omega$ and all $A$, $D_N^{\ge \gamma}(A)=N^{-\eta}\|A\|^2_{\mc{B}}$ if $\gamma$ is chosen sufficiently large. Thus the conclusion holds.
\qed

\section{Central Limit Theorem}\label{sec:central_limit_theorem}
We are now ready to prove Theorem \ref{ThAnQCLT}. The idea is as follows. By studying the variance, it is straightforward to see that for a fixed function $A$,
the central limit theorem for $A$ will hold almost surely. 
Hence, it is immediate that for a countable, dense collection of functions the quenched Central Limit Theorem holds for a full measure set of trajectories. We then extend the Central Limit Theorem from this dense set to every function. 
As the sample variance converges to the annealed variance, to pass to the limit we can then conclude by using some quantitative estimates on the convergence of the sample variance.

\begin{proof}[Proof of Theorem ~\ref{ThAnQCLT}]
First we will give the proof in the $\mc{H}^s$ case, then at the end we will explain the adaptation of the argument to the H\"older case.

We divide the proof into several steps, each of which simplifies what we must check until we have reduced to checking convergence of a single characteristic function at rational frequencies. \vskip2mm

{\bf Step 1.} Let $N_m=m^a$ with $a>1/\eta$  where $\eta$ is the convergence rate in \eqref{Eq2DPolyChar}. It suffices to prove that for a.e.~$\omega$ and each $A\in \cH^s_0(M)$ that
$(N_m)^{-1/2} S_{N_m}(A)$ converges to $\cN(0, \cD(A)).$ \vskip2mm

Indeed, suppose that we have convergence along this subsequence.
Then, given an arbitrary $N$ choose $m$ so that $N_m\leq N<N_{m+1}.$ 
Then letting $S_N=S_N(A)$,
\begin{equation}\label{eqn:decomp_ofS_N}
\frac{S_N}{\sqrt{N}}=\frac{S_{N_m}}{\sqrt{N_m}}+\frac{S_N-S_{N_m}}{\sqrt{N}}+\frac{S_{N_m}}{\sqrt{N_m}}\left(\sqrt\frac{N_m}{N}-1\right)=I+II+III. 
\end{equation}
By Theorem \ref{ThQMix}, given $s>0$, there exists $\beta=\beta(s)>0$ such that for almost every $\omega$ there exists $C(\omega)$ such that
\begin{equation}\label{eqn:upperbound_mixing_slow}
\left|\int A(F^n x) A(F^{n+k} x) dx\right|\leq C(\omega) \|A\|_{\cH^s_0} \min\left(1, n e^{-\beta k}\right). 
\end{equation}
Summing over $N\leq n\leq n+k \leq N_m$ we get
$\int (S_N-S_{N_m})^2\,dx\leq C (N-N_m) \ln N$, so the second term in \eqref{eqn:decomp_ofS_N} converges to zero in probability due to the Chebyshev's inequality.
Also, the third term converges to zero due to the Slutsky's theorem and our assumption that the CLT holds along $N_m.$ Invoking again Slutsky's theorem we see
that the Central Limit Theorem holds along the sequence of all $N$.\vskip2mm

{\bf Step 2.}  It suffices to prove that the quenched Central Limit Theorem holds for a $\cH^s_0$ dense set of functions.
\vskip2mm

Indeed, suppose that there is a dense set of functions $\mc{A}$ and a full measure set of $\Omega$ such that for $\omega\in \Omega$, and $\tA\in \cA$, the (quenched) Central Limit Theorem holds.
Now take arbitrary $A\in \cH^s_0$, and let $h$ be a compactly supported smooth test function
 on $\mathbb{R}$. 
Let $J$ denote the support of $h$. 
We need to show that for $\omega\in \Omega$ that
\begin{equation}
\label{NormTest}
\lim_{N\to \infty} \int h\left(\frac{S_N(A)(x)}{\sqrt{N}}\right)dx=\int_{J} h(u) \gauss_{\cD(A)}(u) du ,
\end{equation}
where $\gauss_{D}$ denotes the density of the normal random variable with zero mean and variance $D$.
Fix $\eps\!>\!0$ and take $\tA\!\in\! \cA$ such that $\DS \|\tA\!-\!A\|_{\cH_0^s} \!\leq\! \eps.$
By Lemma \ref{lem:mcD_est}, $|\cD(A)\!-\!\cD(\tA)|\!\leq\! C_A \eps$. 
Now write
$$ 
\int h\left(\frac{S_N(A)}{\sqrt{N}}\right)dx=\int h\left(\frac{S_N(\tA)}{\sqrt{N}}\right)dx +
\int \left[h\left(\frac{S_N(A)}{\sqrt{N}}\right)- h\left(\frac{S_N(\tA)}{\sqrt{N}}\right)\right]dx. 
$$
Since $\tA\in \cA$, the first term for large $N$ is $\eps$-close to 
$\int_{J} h(u) \gauss_{\cD(\tA)}(u) du $, and hence it is $C_A\eps$ close to 
$\int_{J} h(u) \gauss_{\cD(A)}(u) du $. From the mean value theorem, the second term is smaller in absolute value than
\begin{equation}\label{eqn:second_term_est2}
\|h\|_{C^1} \eps^{1/3} +\|h\|_{C^0} \vol\left(x: \abs{\frac{S_N(A)(x)}{\sqrt{N}}-\frac{S_N(\tA)(x)}{\sqrt{N}}}\geq \eps^{1/3}\right). 
\end{equation}
Note that $\vol$ term in the above line is a random variable depending on $\omega$.
By Proposition \ref{prop:quenched_variance} applied to the function $A-\tA$, and Chebyshev's inequality the above expression is $O(\eps^{1/3})$. 
Since $\eps$ is arbitrary \eqref{NormTest} holds for all $h$, and hence $A$ satisfies the quenched CLT.
\vskip2mm

{\bf Step 3.} Since $\cH^s_0$ contains a countable dense set of $C^{\infty}$ functions, it is enough to show that the quenched CLT holds for a fixed smooth function $A\in \mathcal{H}^s_0$.

\vskip2mm

{\bf Step 4.} For smooth $A\in C^{\infty}_0(M)$, almost surely, the functions $\Phi_{A, N}(\xi)= \int e^{iS_N(A)(x)\xi/\sqrt{N}} dx$ are equicontinuous with respect to $N$.
\vskip2mm

Indeed $\Phi_{A, N}(0)=1$. Taking the first derivative gives:
\[
\partial_\xi\Phi_{A,N}=\int \frac{iS_N(A)(x)}{\sqrt{N}}e^{iS_N(A)(x)\xi/\sqrt{N}}\,dx. 
\]
Note in particular that $\partial_{\xi}\Phi_{A,N}(0)=0$ since $A$ has zero mean, and by Cauchy-Schwarz,
\[
\abs{\partial_{\xi}\Phi_{A,N}(\xi)}\le \left(\int \frac{S_N(A)^2(x)}{N}\,dx\right)^{1/2}.
\]
The above is a random quantity depending on $\omega$. 
By Proposition \ref{prop:quenched_variance}, it follows that for sufficiently large $N$, $\int N^{-1}S^2_N(A)(x)\,dx\le K\|A\|^2_{\mathcal{H}^s_0}$ and hence  $\abs{\partial_{\xi}\Phi_{A,N}}$ is a bounded function. Thus for a fixed $\omega$ the family $\Phi_{A,N}(\xi)$ is equicontinuous. 

\vskip2mm

Combining Steps 1-4 above, we see that it suffices to show for a fixed $C^{\infty}$ smooth $A$ that for almost every $\omega$ that $\Phi_{A,N}(\xi)\to e^{-\cD(A) \xi^2/2}$ for all rational $\xi$ restricted to the sequence $N_m$ from Step~1. Hence it suffices to show that the convergence holds for a fixed (rational) $\xi$.

We now show that for fixed $A$ and $\xi$ that for almost every $\omega$, $\Phi_{A,N}(\xi)\to e^{-\cD(A) \xi^2/2}$. From \eqref{eqn:Ax-Ay}  if 
$B(x,y)\!\!=\!\!A(x)$ then $\bbD(B)\!\!=\!\!\cD(A)$, while if $B(x,y)\!\!=\!\!A(x)\!-\!A(y)$ then  
$\bbD(B)=2\cD(A)$. 
Next let 
\[
\DS Z_N(\omega)=\int e^{i\xi S_N(A)(x)/\sqrt{N}} dx-e^{-\xi^2 \cD(A)/2}.
\]
We claim that $Z_{N_m}$ converges to zero almost surely.
Indeed by \eqref{Eq2DPolyChar}, $\EXP(Z_N)=O(N^{-\eta})$ and, in addition,
\begin{align*}
 \EXP\left(Z_N \bar{Z}_N\right)\!\!=&\EXP\!\!\left(\!\left[\int e^{i\xi S_N(A)(x)/\sqrt{N}} dx\!-\!e^{-\xi^2 \cD(A)/2}\right]\!\!
\left[\int e^{-i\xi S_N(A)(y)/\sqrt{N}} dy\!-\!e^{-\xi^2 \cD(A)/2}\right]\!\right)\!\!\\
=&\EXP\left(\iint e^{i\xi(S_N(A)(x)-S_N(A)(y))/\sqrt{N}}\,dxdy\right)+e^{-\xi^2\mathcal{D}(A)}\\
&-2e^{-\xi^2\mathcal{D}(A)/2}\EXP\left(\mathfrak{R}\int e^{i\xi S_N(A)(x)/\sqrt{N}}\,dx\right)\\
=& e^{-\xi^2\mathcal{D}(A)}+O(N^{-\eta})\!+\!e^{-\xi^2\mathcal{D}(a)}\!-\!2e^{-\xi^2\mathcal{D}(A)/2}(e^{-\xi^2\mathcal{D}(A)/2}+O(N^{-\eta}))
\!\!\\
=& O\left(N^{-\eta}\right).
\end{align*}
Hence as $\E[Z_{N_m}]=O(N^{-\eta})$ and $\E[\abs{Z_{N_m}}^2]=O(N^{-\eta})$ by the above line, it follows from our choice of $N_m$ and Chebyshev's inequality that
 \[
 \lim_{m\to\infty} \int e^{i\xi S_{N_m}(A)(x) \xi} dx=e^{-\xi^2 \cD(A)/2}
 \]
 for almost every $\omega$
 completing the proof of the theorem in the case of Sobolev spaces.

 \noindent \textbf{Adaptation to H\"older Functions.} The main adaptations are in the sequence of steps. Once those are carried out the final paragraph is identical, so we will go through the steps in order and explain what needs to change.
 
     {\bf Step. 1} The only place where we used mixing here was in the application of Theorem~\ref{ThQMix} to obtain equation \eqref{eqn:upperbound_mixing_slow}. Theorem \ref{ThQMix} gives this same conclusion in the H\"older regularity class as well.
     
     {\bf Step. 2} The main difference occurs in this step. 
          It will suffice to show that there is a countable set $\mc{A}$ of $C^{\alpha}$ functions that are $C^{\beta}$-dense in $C^{\alpha}$, and for these functions $A\in \mc{A}$ that the quenched CLT holds. The proof in this step also uses two further facts:
     
         (a) We need to know that $\mc{D}(\mc{A})$ is a continuous function in the $C^{\beta}$-topology on $C^{\alpha}$. This is guaranteed by Lemma \ref{lem:mcD_est}, in the H\"older case. 
        
         (b) Then in order to obtain the estimate in  \eqref{eqn:second_term_est2}, we apply Proposition~\ref{prop:quenched_variance} as before, which gives the needed estimate for H\"older functions.
     
     {\bf Step 3.} In this step, we instead use that $C^{\alpha}(M)$ has a $C^{\beta}$-dense set $\mc{A}$ of $C^{\infty}$ functions due to Proposition \ref{lem:dense_C_alpha}.
     
     {\bf Step 4.} Finally, we need to know that for $A\in \mc{A}$ above, that the family of functions $\Phi_{A,N}(\xi)$, $N\in \N$, are equicontinuous. The only thing used in this computation was Lemma~\ref{prop:quenched_variance}, which holds for H\"older functions as well. 
 
Once all of these steps are completed, we have reduced the proof to showing convergence for a single $C^{\infty}$ function at a single mode $\xi$, which follows exactly as above and is the assumption in the hypotheses, so we can conclude in the H\"older case as well. 
\end{proof}

\printbibliography

\end{document}